\documentclass[oneside,12pt]{amsart} 
\usepackage{eurosym}
\usepackage{amsfonts}
\usepackage{amsmath,amssymb,amsthm,color}
\usepackage{amsopn}
\usepackage{latexsym} 
\usepackage{float}
\usepackage{bm}
\usepackage[utf8]{inputenc}
\usepackage{hhline}
\usepackage[utf8]{inputenc}
\usepackage[british]{babel}
\usepackage{a4wide} 
\usepackage{amsfonts}
\usepackage{pdfsync}
\usepackage{amsmath}
\usepackage{enumerate}
\usepackage{amsthm}
\usepackage{amssymb}
\usepackage{xcolor}
\usepackage{enumerate}
\usepackage[colorlinks=true,linkcolor=blue,citecolor=blue]{hyperref}
\usepackage[all]{xy}
\usepackage{parskip}
\usepackage{multirow}
\usepackage{array}

\theoremstyle{definition} 
\newtheorem{lemma}{Lemma}

\newtheorem{theorem}[lemma]{Theorem}
\newtheorem*{theorem*}{Theorem}

\newtheorem{definition}[lemma]{Definition}
\newtheorem{example}[lemma]{Example}

\newcommand{\R}{\mathbb{R}}

\begin{document}

\title{Invariant spinors on flag manifolds}

\author{Diego Artacho}
\address{D.~Artacho: Department of Mathematics, Faculty of Natural Sciences, Imperial College London, 180 Queen's Gate London SW72AZ, UK}
\email{d.artacho21@imperial.ac.uk}

\author{Uwe Semmelmann} 
\address{U.~Semmelmann: Institut f\"{u}r Geometrie und Topologie, Fachbereich Mathematik, Universit\"{a}t Stuttgart, Pfaffenwaldring 57, 70569 Stuttgart, Germany}
\email{uwe.semmelmann@mathematik.uni-stuttgart.de}

\begin{abstract} 
    In this note, we characterise the existence of non-trivial invariant spinors on maximal flag manifolds associated to complex simple Lie algebras. This characterisation is based on the combinatorial properties of their set of positive roots. We also give some bounds for the dimension of the space of invariant spinors in each case. 
\end{abstract}

\maketitle

\section{Introduction}

The classification of manifolds carrying special spinors has been a significant area of interest in differential geometry, as these objects are often related to $G$-structures and curvature properties. For instance, Wang characterised Riemannian manifolds which carry non-trivial parallel spinors with respect to the Levi-Civita connection \cite{wang}. In particular, some special holonomies in Berger's list can be characterised by the existence of parallel spinors. Moreover, the existence of a non-zero Killing spinor implies that the metric is Einstein. 

It is natural to study spinors parallelised by other connections. For example, the field equations in string theory involve parallel spinors with respect to certain connections with torsion. In this paper, we focus on \textit{invariant} spinors on (a family of) homogeneous spaces, which are precisely the parallel ones with respect to the \textit{canonical connection}, which comes from the homogeneous structure. Invariant spinors are often useful to obtain \emph{explicit} examples of other special spinors, such as Killing or generalised Killing \cite{BFGK,A24}. In particular, all examples of Killing spinors on homogeneous nearly-K{\"a}hler manifolds in \cite{BFGK} are invariant. 

Flag manifolds, which are adjoint orbits of compact forms of complex simple Lie algebras, have a rich structure with interesting spinorial properties. Notably, \textit{maximal} flag manifolds admit a unique invariant spin structure -- see Lemma \ref{lemma:spin}. This paper explores the conditions under which these maximal flag manifolds admit non-trivial invariant spinors. Our main theorem establishes a characterisation of the existence of non-trivial invariant spinors in terms of combinatorial properties of the positive roots of the associated Lie algebra: 

\begin{theorem*}
    Let $\mathfrak{g}_{\mathbb{C}}$ be a complex simple Lie algebra with positive roots $a_1 , \dots , a_r$. Then, the maximal flag manifold associated to $\mathfrak{g}_{\mathbb{C}}$ admits non-trivial invariant spinors if, and only if, there exist $\varepsilon_1 , \dots , \varepsilon_r \in \{ \pm 1 \}$ such that $\sum_{i=1}^{r} \varepsilon_i a_i = 0$. Moreover, the dimension of the space of invariant spinors is equal to the number of different such linear combinations of the set of positive roots which equal zero. In particular, it is always even. 
    \end{theorem*}

We apply this theorem in turn to each complex simple Lie algebra, showing that the existence of non-trivial invariant spinors varies significantly across different algebras - see Theorem \ref{thm:main}. 

Finally, we explore several examples demonstrating that for \textit{general} flag manifolds, as opposed to just \textit{maximal} flag manifolds, the behaviour can vary widely. For instance, complex projective space $\mathbb{CP}^n$, which is a flag manifold of $A_n$, is spin if and only if $n$ is odd. Furthermore, $\mathbb{CP}^{2n+1}$, which is a flag manifold of $C_{n+1}$, exhibits different spinorial properties compared to the maximal flag manifold of $C_{n+1}$. These examples illustrate that, for general flag manifolds, essentially all possible spinorial behaviours can occur. 

This paper sheds some light into the rich interplay between Lie algebras and spin geometry on homogeneous spaces. 

\section{Flag manifolds}

Let us start by recalling the construction of flag manifolds -- see e.g. \cite[Ch. 8]{besse}. Let $\mathfrak{g}_{\mathbb{C}}$ be a complex simple Lie algebra with compact real form $\mathfrak{g}$, and let $G$ be the compact simply connected Lie group with Lie algebra $\mathfrak{g}$. This group has a maximal torus, which we denote by $T$. 
\begin{definition} \label{def:MF}
    The maximal flag manifold associated to $\mathfrak{g}_{\mathbb{C}}$ is 
    \[ 
    \mathrm{MF}\left( \mathfrak{g}_{\mathbb{C}} \right) := G / T \, . 
    \] 
\end{definition}
These manifolds have nice spinorial properties: 
\begin{lemma}\label{lemma:spin}
Let $\mathfrak{g}_{\mathbb{C}}$ be a complex simple Lie algebra, and let $G$ be the simply connected Lie group corresponding to the compact real form of $\mathfrak{g}_{\mathbb{C}}$. Then, $\mathrm{MF}\left(\mathfrak{g}_{\mathbb{C}}\right)$ admits a unique spin structure. Moreover, this spin structure is $G$-invariant. 
\end{lemma}
\begin{proof}
The fact that $\rm{MF} \left( \mathfrak{g}_{\mathbb{C}} \right)$ admits a spin structure follows from the fact that it is diffeomorphic to an adjoint orbit in $\mathfrak{g}$ with trivial normal bundle -- see e.g. \cite[\S 8.23]{besse}. The fact that this structure is unique and $G$-invariant follows from \cite[p. 327]{HS90} -- see also \cite{DKL} and \cite[Prop. 2.1]{AH70}. 
\end{proof}

Note that Definition \ref{def:MF} deals with \textit{maximal} flag manifolds. In fact, we can make a more general definition:  
\begin{definition} \label{def:MF_gen}
    A flag manifold associated to $\mathfrak{g}_{\mathbb{C}}$ is an adjoint orbit of $G$ in $\mathfrak{g}$. 
\end{definition}
These general flag manifolds are quotients of $G$ by the centraliser of a torus, and hence the notion of maximality defined earlier. 

\section{Invariant spinors}

In this section we study the existence of non-trivial invariant spinors on maximal flag manifolds. First, we need a definition: 
\begin{definition}\label{def:sld}
Let $V$ be a vector space. A finite subset $\{ v_1 , \dots , v_r \} \subseteq V$ is said to be \emph{strongly linearly dependent} if there exist $\varepsilon_1 , \dots \varepsilon_r \in \{ \pm 1\}$ such that $\sum_{i=1}^{r} \varepsilon_i v_i = 0$. Such a combination is called a \textit{strong linear combination}. 
\end{definition}

We have a combinatorial characterisation of existence of invariant spinors on maximal flag manifolds in terms of strong linear dependence of positive roots:  
\begin{theorem} \label{thm:comb}
Let $\mathfrak{g}_{\mathbb{C}}$ be a complex simple Lie algebra with positive roots $a_1 , \dots , a_r$. Then, $\mathrm{MF} \left(\mathfrak{g}_{\mathbb{C}} \right)$ admits non-trivial invariant spinors if, and only if, $\{ a_1 , \dots , a_r \}$ is strongly linearly dependent. Moreover, the dimension of the space of invariant spinors is equal to the number of strong linear combinations of the set of positive roots which equal zero. In particular, it is always even. 
\end{theorem}
\begin{proof} 
Let $\mathfrak{g}$ be the compact real form of $\mathfrak{g}_{\mathbb{C}}$, and let $\mathfrak{h} \subseteq \mathfrak{g}$ be a Cartan subalgebra. Then, as a representation of $\mathfrak{h}$, the Lie algebra $\mathfrak{g}$ decomposes as 
\[ 
\mathfrak{g} = \mathfrak{h} \oplus \mathfrak{a}_1 \oplus \cdots \oplus \mathfrak{a}_r \, , 
\] 
where $\dim \left( \mathfrak{a}_i \right) = 2$ and, for each $X \in \mathfrak{h}$, 
\begin{equation} \label{eq:ad}
\mathrm{ad}(X)\rvert_{\mathfrak{a}_i} = a_i(X) \, e^{(i)}_1 \wedge e^{(i)}_2 \, , 
\end{equation} 
where $\left(e^{(i)}_1 , e^{(i)}_2\right)$ is an orthonormal basis of $\mathfrak{a}_i$ with respect to the negative of the Killing form of $\mathfrak{g}$. Note that fixing a partial order in the set of rots, i.e., determining which of $\pm a_i$ is positive, determines an orientation on each $\mathfrak{a}_i$. So, we have a positively-oriented orthonormal basis of the orthogonal complement of $\mathfrak{h}$: 
\[ 
\left( e^{(1)}_1 , e^{(1)}_2 , \dots , e^{(r)}_1 , e^{(r)}_2 \right) \, . 
\] 
We now define, as in \cite{AHL}, 
\begin{align*} 
L := \mathrm{span}_{\mathbb{C}} \left\{ x_j := \frac{1}{\sqrt{2}} \left( e^{(j)}_1 - i e^{(j)}_2 \right) \, \rvert \, 1 \leq j \leq r \right\} \, , \\
L' := \mathrm{span}_{\mathbb{C}} \left\{ y_j := \frac{1}{\sqrt{2}} \left( e^{(j)}_1 + i e^{(j)}_2 \right) \, \rvert \, 1 \leq j \leq r \right\} \, , 
\end{align*} 
and note that the spin representation of the complex Clifford algebra $\mathbb{C}l(2r)$ is isomorphic to the action on $\Lambda^{\bullet} L'$ defined by extending 
\[ 
x_j \cdot \eta := i \sqrt{2} x_j \lrcorner \eta \, , \quad y_j \cdot \eta := i \sqrt{2} y_j \wedge \eta \, ,  
\] 
for $1 \leq j \leq r$ and $\eta \in \Lambda^{\bullet} L'$. Hence, if $G$ is the simply connected Lie group with Lie algebra $\mathfrak{g}$ and $T$ is its maximal torus with Lie algebra $\mathfrak{h}$, the spinor bundle of $\mathrm{MF}\left(\mathfrak{g}_{\mathbb{C}} \right)$ is given by 
\begin{align*} 
G \, \times_{\widetilde{\sigma}} \, &\Lambda^{\bullet} L' \to \mathrm{MF}\left(\mathfrak{g}_{\mathbb{C}} \right) \\
&[g,\psi] \mapsto gT\, , 
\end{align*}
where $\widetilde{\sigma} \colon T \to \rm{Spin}(2r)$ is the lift of the isotropy representation $\sigma \colon T \to \rm{SO}(2r)$ to $\rm{Spin}(2r)$. In particular, spinors correspond to $T$-equivariant maps $G \to  \Lambda^{\bullet} L'$, and \textit{invariant} spinors correspond to \textit{constant} such maps. Invariant spinors, then, can be seen as those elements of $\Lambda^{\bullet} L'$ which are annihilated by the differential action of $\mathfrak{h}$. 

The differential of the isotropy representation is just the adjoint action $\rm{ad}_{\mathfrak{h}}$ of $\mathfrak{h}$ on its reductive complement $\mathfrak{a}_1 \oplus \cdots \oplus \mathfrak{a}_r$. Hence, the action of an element $ X \in \mathfrak{h}$ on $\Lambda^{\bullet} L'$ is given by the action of $\widetilde{\mathrm{ad_{\mathfrak{h}}}}(X) := \rho \left( \rm{ad}_{\mathfrak{h}}(X) \right) \in \mathbb{C}l(2r)$, where $\rho \colon \mathfrak{so}(2r) \to \mathfrak{spin}(2r) \subseteq \mathbb{C}l(2r)$ is the usual map given by $e_i \wedge e_j \mapsto \frac{1}{2} e_i \cdot e_j$.  

By \eqref{eq:ad}, one can easily see that the action of $X \in \mathfrak{h}$ on $\eta \in \Lambda^{\bullet} L'$ is given by 
\[ 
    \widetilde{\mathrm{ad_{\mathfrak{h}}}}(X) \cdot \eta = \frac{1}{2} \sum_{j=1}^{r} a_j(X) \, e^{(j)}_1 \cdot e^{(j)}_2 \cdot \eta \, . 
\] 
As an example of how these calculations work, note that for $1 \leq j_1 < \cdots < j_k \leq r$ and $1 \leq j \leq r$, 
\[
    e^{(j)}_1 \cdot e^{(j)}_2 \cdot \left( y_{j_1} \wedge \cdots \wedge y_{j_k} \right) = \begin{cases} 
      - i \left( y_{j_1} \wedge \cdots \wedge y_{j_k} \right)  & \text{if } j \in \{j_1 , \dots , j_k \} \, , \\
       i \left( y_{j_1} \wedge \cdots \wedge y_{j_k} \right) & \text{if } j \not\in \{j_1 , \dots , j_k \} \, . 
    \end{cases}
\]
In particular, 
\begin{align*}
    2 \widetilde{\mathrm{ad_{\mathfrak{h}}}}(X) \cdot \left( y_{j_1} \wedge \cdots \wedge y_{j_k} \right) = -i \left(  \sum_{j \in \{j_1 , \dots , j_k \}} a_j - \sum_{j \not\in \{j_1 , \dots , j_k \}} a_j \right) (X) \cdot \left( y_{j_1} \wedge \cdots \wedge y_{j_k} \right) \, . 
\end{align*}
The theorem now follows. 
\end{proof}

We now apply Theorem \ref{thm:comb} in turn to each complex simple Lie algebra. 

\begin{theorem}\label{thm:main}
For each of the complex simple Lie algebras, the following holds: 
\begin{enumerate} 
\item $\mathrm{MF} \left( A_n \right)$, $n \geq 1$, admits non-trivial invariant spinors if, and only if, $n$ is even. In this case, the dimension of the space of invariant spinors is at least $2^{\lfloor \frac{n}{2} \rfloor}$.
\item $\mathrm{MF} \left( B_n \right)$, $n \geq 2$, does not admit non-trivial invariant spinors for any value of $n$. 
\item $\mathrm{MF} \left( C_n \right)$, $n \geq 3$, admits non-trivial invariant spinors if, and only if, $n \equiv 0,3 \bmod 4$. In this case, the dimension of the space of invariant spinors is at least $2^{\lfloor \frac{n+1}{4} \rfloor}$. 
\item $\mathrm{MF} \left( D_n \right)$, $n \geq 4$, admits non-trivial invariant spinors if, and only if, $n \geq 4$ and $n \equiv 0,1 \bmod 4$. In this case, the dimension of the space of invariant spinors is at least $2^{\lfloor \frac{n+1}{4} \rfloor}$. 
\item $\mathrm{MF} \left( E_6 \right)$ admits non-trivial invariant spinors. The space of invariant spinors has dimension exactly $13697920$.
\item $\mathrm{MF} \left( E_7 \right)$ does not admit non-trivial invariant spinors. 
\item $\mathrm{MF} \left( E_8 \right)$ admits non-trivial invariant spinors. The space of invariant spinors has dimension at least $369600$.  
\item $\mathrm{MF} \left( F_4 \right)$ admits non-trivial invariant spinors. The space of invariant spinors has dimension exactly $34432$.  
\item $\mathrm{MF} \left( G_2 \right)$ admits a four-dimensional space of invariant spinors. 
\end{enumerate}
\end{theorem}
\begin{proof} 
By Theorem \ref{thm:comb}, we need to discuss strong linear dependence of the set of positive roots of each complex simple Lie algebra. We recall them in Table \ref{table:roots}, and will deal with each case in turn. 

\begin{table}[h]
    \centering
    \begin{tabular}{|c|c l|}
    \hline 
    $\mathfrak{g}_{\mathbb{C}}$ & \multicolumn{2}{|c|}{Positive roots of $\mathfrak{g}_{\mathbb{C}}$} \\ \hline
    \multirow{2}{*}{$A_n$} & $\lambda_i - \lambda_j$ \, , & $1 \leq i < j \leq n$ \\ 
                            & $\lambda_i + \sum_{j=1}^{n} \lambda_j$ \, , & $ 1 \leq i \leq n$ \\ \hline
    \multirow{3}{*}{$B_n$} & $\lambda_i - \lambda_j$ \, , & $1 \leq i < j \leq n$ \\ 
                            & $\lambda_i + \lambda_j$ \, , & $1 \leq i < j \leq n$ \\
                            & $\lambda_i$ \, , & $1 \leq i \leq n$ \\ \hline
    \multirow{3}{*}{$C_n$} & $\lambda_i - \lambda_j$ \, , & $1 \leq i < j \leq n$ \\ 
                            & $\lambda_i + \lambda_j$ \, , & $1 \leq i < j \leq n$ \\
                            & $2\lambda_i$ \, , & $1 \leq i \leq n$ \\ \hline
    \multirow{2}{*}{$D_n$} & $\lambda_i - \lambda_j$ \, , & $1 \leq i < j \leq n$ \\ 
                            & $\lambda_i + \lambda_j$ \, , & $1 \leq i < j \leq n$ \\ \hline
    \multirow{3}{*}{$E_6$} & $\lambda_i - \lambda_j$ \, , & $1 \leq i < j \leq 6$ \\ 
                            & $\lambda_i + \lambda_j + \lambda_k$ \, , & $1 \leq i < j < k \leq 6$ \\ 
                            & $\lambda_1 + \lambda_2 + \lambda_3 + \lambda_4 + \lambda_5 + \lambda_6$ & \\ \hline
    \multirow{3}{*}{$E_7$} & $\lambda_i - \lambda_j$ \, , & $1 \leq i < j \leq 7$ \\ 
                            & $\lambda_i + \lambda_j + \lambda_k$ \, , & $1 \leq i < j < k \leq 7$ \\ 
                            & $\left( \sum_{j=1}^{7} \lambda_j \right) - \lambda_i$ \, , & $1 \leq i \leq 7$ \\ \hline
    \multirow{4}{*}{$E_8$} & $\lambda_i - \lambda_j$ \, , & $1 \leq i < j \leq 8$ \\ 
                            & $\lambda_i + \lambda_j + \lambda_k$ \, , & $1 \leq i < j < k \leq 8$ \\ 
                            & $\left( \sum_{j=1}^{8} \lambda_j \right) + \lambda_i$ \, , & $1 \leq i \leq 8$ \\ 
                            & $\left( \sum_{k=1}^{8} \lambda_k \right) - \lambda_i - \lambda_j$ \, , & $1 \leq i < j \leq 8$ \\ \hline
    \multirow{4}{*}{$F_4$} & $\lambda_i - \lambda_j$ \, , & $1 \leq i < j \leq 4$ \\ 
                            & $\lambda_i + \lambda_j$ \, , & $1 \leq i < j \leq 4$ \\ 
                            & $\lambda_i$ \, ,& $1 \leq i \leq 4$ \\ 
                            & $\frac{1}{2} \left( \lambda_1 \pm \lambda_2 \pm \lambda_3 \pm \lambda_4 \right)$ & \\ \hline
    $G_2$ & $\lambda_1 , \lambda_2,-\lambda_1 - \lambda_2,\lambda_1-\lambda_2, \lambda_1 + 2\lambda_2,2\lambda_1 + \lambda_2$ & \\ \hline
    \end{tabular} 
    \vspace{.3cm}
    \caption{Complex simple Lie algebras and their positive roots (see e.g. \cite[\S 2.14]{Samelson}). The roots of $X_m$ are elements of $\R^m$, and we take $(\lambda_1 , \dots , \lambda_m)$ to be its standard basis. }
    \label{table:roots}
    \end{table}

    \paragraph{$\mathbf{A_n}$} First, we show that for every $k \geq 1$ the set of positive roots of $A_{2k}$ is strongly linearly dependent. For each $1 \leq l < k$, note that
    \begin{align*}
        S_l := \sum_{j=2l+1}^{2k} (-1)^j \left( \lambda_{2l} - \lambda_j \right) - \sum_{j=2l+2}^{2k} (-1)^j \left( \lambda_{2l+1} - \lambda_j \right) = 0 \, . 
    \end{align*}
    
    Letting now $\mu = \sum_{j=1}^{2k} \lambda_j$, observe that 
    \begin{align*}
    T:= \left( \lambda_1 + \mu \right) - \sum_{j=2}^{2k} (-1)^{j} \left( \left( \lambda_1 - \lambda_j \right) + \left(\lambda_j + \mu \right) \right) = 0 \, . 
    \end{align*}
    Finally, then, $\pm S_1 \pm \cdots \pm S_{k-1} \pm T = 0$ is a strong linear combination of the positive roots which equals zero, for every choice of sign at each $\pm$. In particular, the dimension of the space of invariant spinors of $\mathrm{MF}\left(A_{2k}\right)$ is at least $2^{k}$. 
    
    We now turn to the case when $n$ is odd. One can express the positive roots of $A_n$ in terms of the basis $\left( \lambda_1 - \lambda_2 , \lambda_2 - \lambda_3 , \dots , \lambda_{n-1}-\lambda_n , \lambda_{n} + \sum_{j=1}^{n} \lambda_j \right)$ given by a fundamental system. There are exactly $n$ positive roots of $A_n$ whose first coordinate in this basis is $1$, namely $\lambda_1 - \lambda_2 , \dots, \lambda_1 - \lambda_n ,$ and $ \lambda_1 + \sum_{j=1}^{n} \lambda_j$. The rest have vanishing first coordinate. Therefore, if $n$ is odd, the set of positive roots of $A_n$ cannot be strongly linearly dependent. 

    \paragraph{$\mathbf{B_n}$} Every strong linear combination of the positive roots of $B_n$ is a sum of terms of the form $\pm \lambda_i$, and there is an odd number of those with $i=1$.

    \paragraph{$\mathbf{C_n}$} Given any strong linear combination of the positive roots of $C_n$, grouping together the terms $\pm (\lambda_i - \lambda_j)$ and $\pm (\lambda_i + \lambda_j)$, there are $n(n+1)/2$ terms of the form $\pm 2 \lambda_i$. In particular, if $n \equiv 1,2 \bmod 4$, there is an odd number of these, and hence the strong linear combination cannot be zero. 
    
    For $n=4k + \varepsilon \geq 3$, with $\varepsilon \in \{0,3\}$, note that for each $0 \leq l < k$ we have 
    \begin{align*} 
        S'_l := &(\lambda_{4l+1} - \lambda_{4l+2}) + (\lambda_{4l+1} + \lambda_{4l+2})  + \sum_{j=4l+3}^{4k+\varepsilon}\left( (\lambda_{4l+1} + \lambda_j) - (\lambda_{4l+1} - \lambda_j)\right) \\
        &+ (\lambda_{4l+2} - \lambda_{4l+3}) + (\lambda_{4l+2} + \lambda_{4l+3}) - \sum_{j=4l+4}^{4k+\varepsilon}\left( (\lambda_{4l+2} + \lambda_j) - (\lambda_{4l+2} - \lambda_j)\right) \\
        &+ \sum_{j=4l+4}^{4k+\varepsilon}\left( (\lambda_{4l+3} + \lambda_j) - (\lambda_{4l+3} - \lambda_j)\right) - \sum_{j=4l+5}^{4k+\varepsilon}\left( (\lambda_{4l+4} + \lambda_j) - (\lambda_{4l+4} - \lambda_j)\right) \\
        & - 2\lambda_{4l+1} - 2\lambda_{4l+2} - 2\lambda_{4l+3} - 2\lambda_{4l+4} = 0 \, .  
    \end{align*}
    Note that $\pm S'_0 \pm \cdots \pm S'_l$ is a strong linear combination of those positive roots which contain $\lambda_1, \lambda_2, \dots, \lambda_{4l+4}$. Hence, if $\varepsilon = 0$, we are done, as every positive root appears exactly once in $\pm S'_0 \pm \cdots \pm S'_{k-1} = 0$. If $\varepsilon = 3$, note that the remaining roots are strongly linearly dependent, as 
    \begin{align*} 
    T' := &(\lambda_{4k+1} - \lambda_{4k+2}) - (\lambda_{4k+1} - \lambda_{4k+3}) + (\lambda_{4k+2} - \lambda_{4k+3}) + (\lambda_{4k+1} + \lambda_{4k+2}) \\
    &+ (\lambda_{4k+1} + \lambda_{4k+3}) + (\lambda_{4k+2} + \lambda_{4k+3}) - 2\lambda_{4k+1} - 2\lambda_{4k+2} - 2\lambda_{4k+3} = 0 \, , 
    \end{align*} 
    and thus $\pm S'_0 \pm \cdots \pm S'_{k-1} \pm T' = 0$. This proof shows that, when it is not zero, the dimension of the space of invariant spinors is at least $2^{\lfloor \frac{n+1}{4} \rfloor}$. 

    \paragraph{$\mathbf{D_n}$} In a similar way as for $C_n$, any strong linear combination of the positive roots of $D_n$ is the sum of $n(n-1)/2$ terms of the form $\pm 2 \lambda_i$, which shows that for $n \equiv 2,3 \bmod 4$ no such linear combination can be zero. 
    
    For $n = 4k + \varepsilon$, with $\varepsilon \in \{0,1\}$, note that for all $0 \leq l < k$ we have 
    \begin{align*} 
        S''_{l} := &\sum_{j=4l+2}^{4k+\varepsilon} (-1)^j \left( (\lambda_{4l+1} - \lambda_j) - (\lambda_{4l+1} + \lambda_j) \right) \\
        &+ (\lambda_{4l+2} - \lambda_{4l+3}) + (\lambda_{4l+2} + \lambda_{4l+3})
    - \sum_{j=4l+4}^{4k+\varepsilon} (-1)^j \left(  (\lambda_{4l+2} - \lambda_j) - (\lambda_{4l+2} + \lambda_j) \right) \\
        &- (\lambda_{4l+3} - \lambda_{4l+4}) - (\lambda_{4l+3} + \lambda_{4l+4}) + \sum_{j=4l+5}^{4k+\varepsilon} (-1)^j \left(  (\lambda_{4l+3} - \lambda_j) - (\lambda_{4l+3} + \lambda_j) \right) \\
        & - \sum_{j=4l+5}^{4k+\varepsilon} (-1)^j \left(  (\lambda_{4l+4} - \lambda_j) - (\lambda_{4l+4} + \lambda_j) \right)  = 0 \, ,  
    \end{align*} 
    As $\varepsilon \in \{0,1\}$, all the positive roots appear exactly once in $ \pm S''_0 \pm \cdots \pm S''_{k-1} = 0$. It is clear, then, that the number of strong linear combinations of the roots that equal zero is at least $2^k$. 
  
    \paragraph{$\mathbf{E_6}$} We exhibit an explicit strong linear combination of the positive roots of $E_6$ that does the job: 
    \begin{align*} 
        &(\lambda_1 + \lambda_2 + \lambda_3 + \lambda_4 + \lambda_5 + \lambda_6) - (\lambda_1 -\lambda_2) + (\lambda_1 - \lambda_3) + (\lambda_1 - \lambda_4) + (\lambda_1 - \lambda_5) - (\lambda_1 - \lambda_6) \\
        &+ (\lambda_2 - \lambda_3) + (\lambda_2 - \lambda_4) - (\lambda_2 - \lambda_5) + (\lambda_2 - \lambda_6) - (\lambda_3 - \lambda_4) + (\lambda_3 - \lambda_5) - (\lambda_3 - \lambda_6) \\
        &- (\lambda_4 - \lambda_5) - (\lambda_4 - \lambda_6) + (\lambda_5 - \lambda_6) - (\lambda_1 + \lambda_2 + \lambda_3) - (\lambda_1 + \lambda_2 + \lambda_4) - (\lambda_1 + \lambda_2 + \lambda_5) \\
        & - (\lambda_1 + \lambda_2 + \lambda_6) - (\lambda_1 + \lambda_3 + \lambda_4) + (\lambda_1 + \lambda_3 + \lambda_5) + (\lambda_1 + \lambda_3 + \lambda_6) + (\lambda_1 + \lambda_4 + \lambda_5) \\
        &+ (\lambda_1 + \lambda_4 + \lambda_6) - (\lambda_1 + \lambda_5 + \lambda_6) + (\lambda_2 + \lambda_3 + \lambda_4) - (\lambda_2 + \lambda_3 + \lambda_5) + (\lambda_2 + \lambda_3 + \lambda_6) \\
        &+ (\lambda_2 + \lambda_4 + \lambda_5) - (\lambda_2 + \lambda_4 + \lambda_6) - (\lambda_2 + \lambda_5 + \lambda_6) + (\lambda_3 + \lambda_4 + \lambda_5) + (\lambda_3 + \lambda_4 + \lambda_6) \\
        &- (\lambda_3 + \lambda_5 + \lambda_6) - (\lambda_4 + \lambda_5 + \lambda_6)  = 0 \, . 
    \end{align*}
    The value of the dimension of the space of invariant spinors is obtained with the aid of a computer.  

    \paragraph{$\mathbf{E_7}$} In any strong linear combination of the positive roots of $E_7$, expressed as a sum of terms of the form $\pm \lambda_i$, there are $27$ terms of the form $ \pm \lambda_1$, so the strong linear combination cannot be zero.
    \paragraph{$\mathbf{E_8}$} Define $\nu := \sum_{i=1}^{8} \lambda_i$, and note that: 
    \begin{align*} 
        & \sum_{1 \leq i < j < k \leq 8} (\lambda_i + \lambda_j + \lambda_k) = \sum_{1 \leq i < j \leq 8} (-\lambda_i - \lambda_j + \nu) = 21 \nu \, , \\
        & \sum_{j=2}^{8} (-1)^j (\lambda_1 - \lambda_j) = - \sum_{j=1}^{8} (-1)^j (\lambda_j + \nu) \, . 
    \end{align*} 
    Therefore, the set 
    \begin{align*} 
&\{ \lambda_1 - \lambda_j \colon 2 \leq j \leq 8 \} \cup \{ \lambda_i + \nu \colon 1 \leq i \leq 8 \} \\ &\cup \{ \lambda_i + \lambda_j + \lambda_k \colon 1 \leq i < j < k \leq 8 \} \cup \{ -\lambda_i - \lambda_j + \nu \colon 1 \leq i < j \leq 8\}
    \end{align*}
    is strongly linearly dependent. It is enough, then, to show that the set $\{ \lambda_i - \lambda_j \, \rvert \, 2 \leq i < j \leq 8 \}$ is as well. But this follows from the proof of the case $A_{2k}$ for $k=3$, by shifting the indices and taking $\mu = \lambda_8$. Using a computer, we obtain the claimed lower bound for the dimension of the space of invariant spinors: $2 \times 70 \times 2640 = 369600$. 
    \paragraph{$\mathbf{F_4}$}
    We exhibit an explicit strong linear combination of the positive roots of $F_4$ that does the job: 
    \begin{align*}
        & \lambda_1 - \lambda_2 - \lambda_3 - \lambda_4 - (\lambda_1 - \lambda_2) - (\lambda_1 - \lambda_3) - (\lambda_1 - \lambda_4) - (\lambda_2 - \lambda_3) - (\lambda_2 - \lambda_4) - (\lambda_3 - \lambda_4) \\
        &- (\lambda_1 + \lambda_2) + (\lambda_1 + \lambda_3) - (\lambda_1 + \lambda_4) + (\lambda_2 + \lambda_3) + (\lambda_2 + \lambda_4) - (\lambda_3 + \lambda_4) \\
        &+ \frac{1}{2}(\lambda_1 + \lambda_2 + \lambda_3 + \lambda_4) - \frac{1}{2}(\lambda_1 - \lambda_2 + \lambda_3 + \lambda_4) + \frac{1}{2}(\lambda_1 + \lambda_2 - \lambda_3 + \lambda_4) \\
        &+ \frac{1}{2}(\lambda_1 + \lambda_2 + \lambda_3 - \lambda_4) + \frac{1}{2}(\lambda_1 - \lambda_2 - \lambda_3 + \lambda_4) + \frac{1}{2}(\lambda_1 - \lambda_2 + \lambda_3 - \lambda_4) \\
        &+ \frac{1}{2}(\lambda_1 + \lambda_2 - \lambda_3 - \lambda_4) + \frac{1}{2}(\lambda_1 - \lambda_2 - \lambda_3 - \lambda_4) = 0 \, . 
    \end{align*}
    The claimed dimension is easily obtained with the aid of a computer. 

    \paragraph{$\mathbf{G_2}$} There are exactly four strong linear combinations of the positive roots of $G_2$ that equal zero: 
    \begin{align*} 
        &\pm \left( \lambda_1 + \lambda_2 + (-\lambda_1 - \lambda_2) - (\lambda_1 - \lambda_2) - (\lambda_1 + 2\lambda_2) + (2\lambda_1 + \lambda_2) \right) = 0 \, , \\
        &\pm \left( \lambda_1 + \lambda_2 + (-\lambda_1 - \lambda_2) + (\lambda_1 - \lambda_2) + (\lambda_1 + 2\lambda_2) - (2\lambda_1 + \lambda_2) \right) = 0 \, .
    \end{align*}
    This finishes the proof of Theorem \ref{thm:main}. 
\end{proof}

Until now, we have only considered \textit{maximal} flag manifolds, and one might ask whether similar results can be obtained for \textit{general} flag manifolds. Given a complex simple Lie algebra, its flag manifolds can behave very differently, as we show now. The following examples show that, when it comes to the spinorial properties we have been considering so far, general flag manifolds can have a wild behaviour. 

\begin{example} (Non-spin flag manifolds)
Complex projective space $\mathbb{CP}^n$ is a flag manifold of $A_n$, and this space is spin if, and only if, $n$ is odd. 
\end{example}
\begin{example} (Two flag manifolds of the same Lie algebra with all possible \emph{spinorial behaviours})
Complex projective space $\mathbb{CP}^{2n+1}$ is also a flag manifold of $C_{n+1}$. We know that this space has invariant spinors if, and only if, $n$ is odd -- see \cite{AH24}. However, by Theorem \ref{thm:main}, $\mathrm{MF}\left(C_{n+1}\right)$ has invariant spinors if, and only if, $n \equiv 2,3 \bmod 4$. In particular: 
\begin{itemize}
\item $\mathbb{CP}^{3}$ is a (non-maximal) flag manifold of $B_2$ which admits non-trivial invariant spinors, however the maximal flag manifold of $B_2$ does not. 
\item $\mathbb{CP}^{5}$ is a (non-maximal) flag manifold of $C_3$ which does not admit non-trivial invariant spinors, however, the maximal flag manifold of $C_3$ does. 
\item $\mathbb{CP}^{7}$ is a (non-maximal) flag manifold of $C_4$ which admits non-trivial invariant spinors, and the maximal flag manifold of $C_4$ does as well. 
\item $\mathbb{CP}^{9}$ is a (non-maximal) flag manifold of $C_5$ which does not admit non-trivial invariant spinors, and neither does the maximal flag manifold of $C_5$. 
\end{itemize}
\end{example}

\section*{Acknowledgements}
D. Artacho is funded by the UK Engineering and Physical Sciences Research Council (EPSRC), grant EP/W523872. The authors are grateful to the anonymous referees for their comments and suggestions, and to {\'A}lvaro Vicente for his assistance with the computational aspect of the paper.  

\bibliographystyle{alphaurl}
\bibliography{references.bib}

\end{document}